\documentclass[11pt]{article}
\usepackage{amsfonts,amsmath,amsthm,amssymb,enumerate,color}

\allowdisplaybreaks 

\usepackage[makeroom]{cancel}
\usepackage{amsthm}
\usepackage[shortlabels]{enumitem}
\setlist[enumerate,1]{wide, labelindent=0pt,label={\upshape(\roman*)}}
\usepackage{float}
\usepackage{authblk}
\usepackage{bbm}
\usepackage{mathtools}

\usepackage{hyperref}

\providecommand{\MR}{\relax\ifhmode\unskip\space\fi MR }

\providecommand{\href}[2]{#2}

\numberwithin{equation}{section}
\newcommand{\beq}{\begin{equation}}
\newcommand{\eeq}{\end{equation}}
\newcommand{\bea}{\begin{eqnarray}}
\newcommand{\eea}{\end{eqnarray}}
\newcommand{\beas}{\begin{eqnarray*}}
\newcommand{\eeas}{\end{eqnarray*}}

\newtheorem{theorem}{Theorem}[section]

\newtheorem{assumption}[theorem]{Assumption}
\newtheorem{definition}[theorem]{Definition}

\newtheorem{corollary}[theorem]{Corollary}
\newtheorem{lemma}[theorem]{Lemma}
\newtheorem{remark}[theorem]{Remark}
\newtheorem{example}[theorem]{Example}
\newtheorem{examples}[theorem]{Examples}
\newtheorem{foo}[theorem]{Remarks}

\newcommand{\bM}{\mathbb M}

\newcommand{\Ho}{\mathcal H}
\newcommand{\V}{\mathcal V}

\newcommand{\M}{\mathbb M}

\newcommand{\Lie}{\mathcal{L}}
\newcommand{\Tor}{\mathrm{Tor}}
\newcommand{\hor}{\mathcal{H}}
\newcommand{\ver}{\mathcal{V}}

\mathtoolsset{showonlyrefs}

\usepackage{todonotes}

\title{Topology and  bottom spectrum of transversally negatively curved foliations}
\author{Fabrice Baudoin\footnote{Research partially funded by Villum Fonden through a Villum Investigator grant and by grant
10.46540/4283-00175B from Independent Research Fund Denmark}}

\affil{Department of Mathematics, Aarhus University \\
\texttt{ fbaudoin@math.au.dk}}

\date{August 2025}

\begin{document}

\maketitle

\begin{abstract}
We show that for any  Riemannian foliation with closed leaves, trivial leaf holonomy and a simply connected and negatively curved leaf space, the normal exponential map of a  leaf is a diffeomorphism.  As an application, if the leaves are furthermore minimal submanifolds,  we give a sharp estimate for the bottom of the spectrum of such a  Riemannian manifold. Our proof of the spectral estimate  also yields an estimate for the bottom of the spectrum of the horizontal Laplacian.
\end{abstract}

\tableofcontents

\section{Introduction}

The Cartan-Hadamard is a fundamental theorem in Riemannian geometry and one of the most powerful and elegant results relating curvature to  the global geometry of a space. The theorem states that the universal cover of complete Riemannian manifolds of non-positive sectional curvature is diffeomorphic to a Euclidean space via the exponential map at any point, see \cite[Section 3.87]{GallotHulinLafontaine2004}. Several generalizations of that result have been given in the context of metric spaces following the pioneering work \cite{Gromov1987}  by M. Gromov, see \cite[Chapter II.4]{BridsonHaefliger1999}. For Riemannian foliations, J. Hebda proved in \cite{Hebda} that if a complete bundle-like metric has nonpositive transverse sectional curvature, then the universal cover of the manifold fibers over a complete simply connected Riemannian manifold of nonpositive curvature, each fiber being diffeomorphic to the universal cover of a leaf. In this paper we give an analogue of the Cartan-Hadamard theorem for Riemannian foliations which is formulated in terms of the normal exponential map of a leaf:

\begin{theorem}\label{Cartan intro}
Let $\M$ be a complete Riemannian manifold equipped with a Riemannian foliation with bundle-like metric whose leaves are closed and have trivial holonomy. Assume that the leaf space of this foliation is simply connected and that the transverse sectional curvature is bounded above by $-K$ where $K \ge 0$. Then, for any leaf the normal bundle of that leaf is diffeomorphic to $\M$ through the normal exponential map.
\end{theorem} 
Denoting by $\mathfrak{K}$ the sectional curvature of the Levi-Civita connection on $\M$, the condition that the transverse sectional curvature be bounded above by $-K$ is  equivalent  to
\[
\mathfrak{K}(X,Y) \le -K -\frac{3}{4} \left|  [X,Y]_\mathcal{V} \right|^2
\]
 at every point $x \in \M$ and for every $X,Y \in \Gamma (\mathcal{V}^\perp)$ that are orthonormal at $x$, where $\mathcal{V}$ is the bundle of directions tangent to the leaves.   Our proof of Theorem \ref{Cartan intro} follows the classical approach to the original Cartan–Hadamard theorem but introduces new techniques which are specific to the foliation setting. We first show that, under our conditions, there are no focal points along geodesics perpendicular to the leaves. To achieve this, we extend the theory of  $\mathcal{F}$-Jacobi fields, which arise as variation fields of such geodesics. While the study of these fields originates in the seminal work of E. Heintze and H. Karcher \cite{MR0533065} (see also \cite[Section 1.6]{Gromoll}), a key innovation in our approach is to express these fields as solutions of a simple equation involving a connection $\nabla$ that, unlike the Levi-Civita connection, is adapted to the geometry of the foliation.  Once the absence of conjugate points is proved, we employ covering theory and metric geometry techniques, tracing back to W. Ambrose \cite{Ambrose} to conclude.
 
 In the second part of the paper we give applications of Theorem \ref{Cartan intro} to spectral properties.  In an influential 1970's paper \cite{McKean}, H. P. McKean proved that the bottom of the spectrum of a simply connected, $n$-dimensional manifold with  sectional curvature bounded above by $-K$, $K>0$, is at least $\frac{1}{4} (n-1)^2K$. This bound is sharp since it is an equality on the $n$-dimensional hyperbolic space $H_{n,K}$ of constant curvature $-K$. 
McKean's original argument is simple and ingenious: He first observes that the statement is equivalent to the functional inequality
\[
\frac{1}{4} (n-1)^2K \int f^2 \le \int |\nabla f|^2.
\]
Using then the Cartan-Hadamard theorem which allows to globally write the metric in exponential normal coordinates, this functional inequality is seen to be equivalent to a functional inequality in $\mathbb{R}^n$ which is eventually proved using manipulations of index forms. Since this work, the study of the spectrum of negatively curved manifolds has remained a vibrant area of research up to nowadays, see for instance \cite{MR4617797}. In our context we prove the following theorem which is our main second result:

\begin{theorem}\label{main theorem}
Let $\M$ be a complete Riemannian manifold equipped with a Riemannian foliation of codimension $d \ge 2$ with bundle-like metric whose leaves are closed, minimal and have trivial holonomy. Assume that the leaf space of $\M$ is simply connected and that the transverse sectional curvature is bounded above by $-K$ where $K>0$. Then, the bottom of the spectrum of $\M$ is at least $\frac{1}{4} (d-1)^2K$.
\end{theorem}

  We note that the constant $\frac{1}{4} (d-1)^2K$ is sharp since for the product manifold $\M=H_{d,K} \times \mathbb{R}^m$, $m \ge 1$, the bottom of the spectrum is exactly $\frac{1}{4} (d-1)^2K$. We also note that, since the leaves are closed and have trivial holonomy, the foliation is given by the fibers of a Riemannian submersion onto the leaf space (see Remark \ref{leaf space manifold}), so that this lower bound can also be derived from the combination of the usual McKean's theorem with the existing results \cite[Theorem 1.1]{Polymerakis} and \cite{CavalcanteManfio2018}. The proof given here is different and also yields the estimate for the horizontal Laplacian.

The paper is organized as follows. In Section 2, we present the necessary background on foliations and collect several preliminary results which will later be useful. In particular, we introduce a connection $\nabla$ which is a central character in our analysis. In Section 3, we develop further the theory of $\mathcal{F}$-Jacobi fields.  A highlight of the section is the following lemma which provides a formula for the Hessian, with respect to the connection $\nabla$, of the distance to a leaf.

\begin{lemma}[See Lemma \ref{Hessian comparison}]
Let $\M$ be a complete $n$-dimensional Riemannian manifold equipped with a Riemannian foliation with bundle-like metric and closed leaves. Let $\gamma : [0,\rho] \to \M$ be a unit-speed geodesic which is perpendicular to the leaves and $r$ be the distance function to the leaf that contains $\gamma(0)$. Assume that, for some $\varepsilon>0$, the extension of $\gamma$ as a geodesic on $[0,\rho+\varepsilon]$ satisfies $r(\gamma(t))=t$ for every $t \in [0,\rho+\varepsilon]$. Then, at the point $\gamma(\rho)$, one has for every $X \in T_{\gamma(\rho)} \M$ with $X \perp \dot{\gamma}(\rho)$, 
\begin{align}\label{Index hess}
\mathrm{Hess}^{\nabla}(r)(X,X) = \int_0^{\rho}(| \nabla_{\dot\gamma} V_\mathcal{H}|^2 - \langle R(V_\mathcal{H},\dot\gamma) \dot\gamma,V_\mathcal{H} \rangle )dt=\left\langle V_\mathcal{H} (\rho),  \nabla_{\dot\gamma} V_\mathcal{H} (\rho) \right\rangle
\end{align}
where $V$ denotes the $\mathcal{F}$-Jacobi field along $\gamma$ such that $V_\mathcal{H}(0) = 0$ and $V(\rho) = X$, $R$ is the Riemann curvature tensor of $\nabla$, and $V_\mathcal{H}$ denotes the projection of $V$ onto the set of directions which are perpendicular to the leaves.
\end{lemma}
It is remarkable that the Hessian formula \eqref{Index hess}  applies in any Riemannian foliation with bundle-like metric and that it does not require the leaves to be minimal. However, let us note that the Hessian is computed with respect to the connection $\nabla$ and while $\mathrm{Hess}^{\nabla}(r)$ vanishes on vector fields fields tangent to the leaves, in our generality the same does not hold when the Hessian is computed for the Levi-Civita connection $D$. Actually, we will see that the  trace of $\mathrm{Hess}^{D}-\mathrm{Hess}^{\nabla}$ along the leaves  exactly yields the mean curvature vector of the leaves. The hypothesis of the minimality of the leaves will therefore come into play when applying \eqref{Index hess} to estimate the Laplacian of $r$.

 Section 4 is the heart of the matter: We  prove Theorem \ref{Cartan intro} and that, under our assumptions, the distance to a leaf is smooth everywhere but on the leaf itself. 
  
In  Section 5, we prove a Laplacian comparison for the distance to a leaf and give the proof of Theorem \ref{main theorem} using the Laplacian comparison theorem, the Rayleigh quotient characterization of the bottom of the spectrum and an integration by parts. The argument also yields an estimate for the bottom of the spectrum of the horizontal Laplacian.

\

\textbf{Acknowledgments.}  I am grateful to the reviewer for a careful reading and suggesting  the  proof of Theorem \ref{poincare inequality}. It greatly simplifies the argument given in a previous versions of this work.

\section{Preliminaries}
\subsection{Riemannian foliations} \label{ssec:foliation}

We consider a smooth, connected, $n+m$ dimensional Riemannian manifold $(\M,g)$ which is equipped with a  Riemannian foliation $\mathcal F$ of  rank $m$. We assume $m \ge 1$ and $n \ge 2$. For $x \in \M$ we denote by $\mathcal F_x$ the $m$-dimensional leaf going through $x$. The sub-bundle $\mathcal{V}$ of the tangent bundle $T\M$ formed by vectors tangent to the leaves is referred  to as the set of \emph{vertical directions}. The sub-bundle $\mathcal{H}$ which is normal to $\mathcal{V}$ is referred to as the set of \emph{horizontal directions} or \emph{transverse directions}.  Any vector $u \in T_x\M$ can therefore be decomposed as 
\[
u=u_\mathcal{H} +u_\mathcal{V}
\]
where $u_\mathcal{H}$ (resp. $u_\mathcal{V}$) denotes the orthogonal projection of $u$ on $\mathcal{H}_x$  (resp. $\mathcal{V}_x$).  The set of smooth sections of $\mathcal{H}$ (resp. $\V$) is denoted by $\Gamma(\mathcal{H})$ (resp. $\Gamma(\mathcal{V})$). For convenience, we will often use the notations $\left\langle \cdot,\cdot \right\rangle=g(\cdot,\cdot)$ and $ | \cdot |^2=g(\cdot,\cdot)$.

The literature on Riemannian foliations is vast and we refer for instance to the monograph  \cite{Gromoll} or to the classical reference  \cite{Tondeur} and its rich bibliography.  We make the following first set of assumptions which will be in force throughout the paper:

\begin{assumption} \label{assump total}
\
\begin{enumerate}
\item The metric $g$ is complete and bundle-like for the foliation $\mathcal{F}$.
\item The leaves of the foliation are closed subsets of $\M$. 
\end{enumerate}
\end{assumption}

We recall, see \cite[Theorem 5.19]{Tondeur}, that the bundle-like property of the metric $g$ can for instance be characterized by the property  that for every $X \in \Gamma (\Ho)$, $Z \in \Gamma (\V)$,
 \begin{equation}\label{bundle-like}
 \Lie_Z g (X,X)=0,
 \end{equation}
where $\Lie$ denotes the Lie derivative. Under that condition, a geodesic which is perpendicular to a leaf at one point is perpendicular to all the leaves it intersects, see  \cite{MR0107279}.

%
%
%

\subsection{Leaf space}

Given the foliation $\mathcal F$ on $\M$, the leaf space of $\mathcal{F}$  is defined as $\M / \mathcal{F} :=\left\{  \mathcal{F}_x , x \in \M \right\}$.  Denoting by $d$  the Riemannian distance associated with the Riemannian metric $g$, one can equip  $\M / \mathcal{F}$ with the  distance
\[
d_{\M/\mathcal F} (\mathcal {F}_x ,\mathcal{F}_y)=\inf \left\{ d(z_1,z_2):  z_1 \in \mathcal{F}_x, z_2 \in \mathcal{F}_y \right\} 
\]
and the projection map $x \to \mathcal{F}_x$ is open and 1-Lipshitz continuous, see \cite{Hermann}.


Let $\rho \in \mathbb{R}$, $\rho>0$. Given a continuous curve $\sigma: [0,\rho] \to  \M / \mathcal{F}$ we define the length $L(\sigma) \in [0,\infty]$ of $\sigma$ by
\[
L(\sigma)=\sup \sum_{i=1}^k d_{\M/\mathcal F} \left( \sigma(t_{i-1}), \sigma(t_i) \right),
\]
where the supremum is taken over all the subdivisions $0\le t_0 \le \cdots \le t_k \le \rho$. The curve $\sigma$ is called rectifiable if its length is finite and is called a unit-speed minimizing geodesic if for every $0 \le s \le t \le \rho$, $d_{\M/\mathcal F} (\sigma(s),\sigma(t))=|t-s|$.

\begin{lemma}\label{geodesic space}
The metric space $\M / \mathcal{F}$ is geodesic, i.e. for every $p_1,p_2 \in \M / \mathcal{F}$, there exists a unit-speed minimizing geodesic $\sigma$ connecting $p_1$ to $p_2$.
\end{lemma}

\begin{proof}
Let $x_1 \in \M$ such that $\mathcal{F}_{x_1}=p_1$. Since $p_2$ is closed in $\M$ there exists $x_2 \in p_2$ such that $d_{\M/\mathcal F}(p_1,p_2)=d( x_1,x_2) =\inf_{y \in p_2} d(x_1,y)$, see \cite[Lemma 4.3]{Hermann}. Now, the metric $g$ is complete, from Hopf-Rinow's theorem there exists therefore a smooth unit-speed minimizing geodesic $\gamma: [0,\rho] \to \M$ such that $\gamma(0)=x_1$ and $\gamma(\rho)=x_2$ where $\rho=d(x_1,x_2)$. We define $\sigma(t)=\mathcal{F}_{\gamma(t)}$. We then have
\begin{align*}
L(\sigma)&=\sup \sum_{i=1}^k d_{\M/\mathcal F} \left( \sigma(t_{i-1}), \sigma(t_i) \right)  \le \sup \sum_{i=1}^k d \left( \gamma(t_{i-1}), \gamma(t_i) \right)= \rho.
\end{align*} 
But since $\sigma$ connects $p_1$ to $p_2$, we also have $\rho=d_{\M/\mathcal F}(p_1,p_2) \le L(\sigma)$. Thus, $L(\sigma)=\rho$. It easily follows that $\sigma$ is a unit-speed minimizing geodesic.
\end{proof}

\begin{remark}\label{leaf space manifold}
 In general, the leaf space $\M / \mathcal{F}$ is not a manifold but a Riemannian orbifold,  see \cite{Reinhart2}. However, if the leaf holonomy is trivial, then there is a smooth Riemannian structure on $\M / \mathcal{F}$ and the projection map $x \to \mathcal{F}_x$ is then a Riemannian submersion, see \cite{Escobales} and  \cite{Hermann}. In that case, $d_{\M/\mathcal F}$ is the Riemannian distance of $\M / \mathcal{F}$: indeed, a Riemannian submersion does not increase the length of curves and, since $g$ is complete, every piecewise smooth curve in $\M / \mathcal{F}$ has horizontal lifts in $\M$, which have the same length.
\end{remark}

\subsection{The connection $\nabla$}

The Levi-Civita connection $D$ associated with the Riemannian metric $g$ is in general not well adapted to study foliations since the horizontal and vertical bundle might not be $D$-parallel.  There is a more natural connection $\nabla$   that respects the foliation structure, see \cite{baudoin2019comparison} and  \cite{Hladky}. This connection can be described as follows. We first define a $(2,1)$ tensor $ C$ by the formula:
\begin{equation}
\left\langle C_X Y, Z\right\rangle = \frac{1}{2}(\mathcal{L}_{X_\hor}g)(Y_\ver,Z_\ver).
\end{equation}
We note that the leaves of the foliations are totally geodesic if and only if  one identically has $C=0$, as follows for instance from \cite[Theorem 5.23]{Tondeur}.
The following properties hold:
\begin{equation}
C_\ver \ver = 0,\qquad C_{\ver} \hor =0, \qquad C_\hor \hor = 0,\qquad C_{\hor} \ver \subseteq \ver.
\end{equation}

The connection $\nabla$ is then defined in terms of the Levi-Civita one $D$ by
\[
\nabla_X Y =
\begin{cases}
 ( D_X Y)_{\mathcal{H}} , &X,Y \in \Gamma(\mathcal{H}), \\
  [X,Y]_{\mathcal{H}},  &X \in \Gamma(\mathcal{V}),\ Y \in \Gamma(\mathcal{H}), \\
  [X,Y]_{\mathcal{V}}+C_X Y,  &X \in \Gamma(\mathcal{H}),\ Y \in \Gamma(\mathcal{V}), \\
 (D_X Y)_{\mathcal{V}} , &X,Y \in \Gamma(\mathcal{V}).
\end{cases}
\]
A fundamental property of $\nabla$ is that $\Ho$ and $\V$ are parallel. One can check that this connection is metric, i.e.  satisfies $\nabla g =0$ and has a torsion tensor  $\Tor$  which is  given by
\begin{equation}
\Tor(X,Y) = \begin{cases}
- [X,Y]_\mathcal{V} & X,Y \in \Gamma(\hor), \\
C_X Y  &  X \in \Gamma(\hor), Y \in \Gamma(\V), \\
0 & X,Y \in \Gamma(\ver).
\end{cases}
\end{equation}
We notice that $\Tor(X,Y)$ is always vertical, a fact which will be repeatedly used in the sequel.
%
%

For $Z \in \Gamma(\V)$, there is a  unique skew-symmetric endomorphism  $J_Z:T_x\M \to T_x\M$ such that for all vector fields $X$ and $Y$,
\begin{align}\label{Jmap}
\left\langle J_Z (X),Y\right\rangle= \left\langle Z,\Tor (X,Y) \right\rangle.
\end{align}
If $Z\in \Gamma (\Ho)$, from \eqref{Jmap} we set $J_Z=0$. 

Let us denote by $D$ the Levi-Civita connection. From \cite[Theorem 9.24]{Besse} if $X,Y\in \Gamma (\Ho)$, then
\[
(D_X Y)_\V=-(D_Y X)_\V=\frac{1}{2} [X,Y]_\V,
\]
and the relation between the Levi-Civita connection  and the connection $\nabla$ is given by the formula

\begin{equation}\label{Levi-Civita}
\nabla_X Y=D_X Y+\frac{1}{2} \Tor (X,Y)-\frac{1}{2} J_XY-\frac{1}{2} J_Y X, \quad X,Y \in \Gamma (\M).
\end{equation}
Formula \ref{Levi-Civita} can be proved using the characterization of the Levi-Civita connection as the unique metric and torsion free connection.

It is then a trivial consequence of \eqref{Levi-Civita} that the equation of geodesics is
\begin{align}\label{geodesic1}
\nabla_{\dot\gamma} \dot\gamma+ J_{\dot\gamma} \dot\gamma =0.
\end{align}
%

The Riemann curvature tensor of the  connection $\nabla$ is denoted by $R$ and is defined as
\begin{align}\label{riem curv}
R(X,Y)Z=\nabla_X\nabla_Y Z -\nabla_Y \nabla_X Z -\nabla_{[X,Y]} Z, \quad X,Y,Z \in \Gamma (\M).
\end{align}
Let us first  notice that since $\nabla$ preserves  $\mathcal H$ and $\mathcal V$, so does $R(X,Y)$ for every  $X,Y \in \Gamma (\M)$. We have then the following lemma.

\begin{lemma}\label{hor-ver}
For every $X \in \Gamma(\Ho)$ and $Y \in  \Gamma(\V)$, $ R(Y,X)X=0$.
\end{lemma}

\begin{proof}
Since the torsion tensor $\Tor$ always yields a vertical field, it follows from the first Bianchi identity for connections with torsion, see for instance \cite[Theorem 1.24]{Besse}, that for every  $X,Y,Z \in \Gamma (\M)$,
\begin{align}\label{Bianchi}
R(X,Y)Z+R(Y,Z)X+R(Z,X)Y \in  \Gamma (\V).
\end{align}
Then, since the connection $\nabla$ is metric, we have for every  $X,Y,Z,W \in \Gamma (\M)$,
\begin{align}\label{metric connection}
\left\langle R(X,Y)Z,W \right\rangle = -\left\langle R(X,Y)W,Z \right\rangle.
\end{align}
For  $X \in \Gamma(\Ho)$, $Y \in  \Gamma(\V)$ and $W \in \Gamma(\Ho)$, we have then  using \eqref{metric connection} and \eqref{Bianchi}
\begin{align*}
\left \langle R(Y,X)X,W \right\rangle&=-\left \langle R(Y,X)W,X \right\rangle =\left \langle R(X,W)Y,X \right\rangle+\left \langle R(W,Y)X,X \right\rangle=0.
\end{align*}
Since $W$ is arbitrary $R(Y,X)X$ is therefore vertical, but it is also horizontal because $X$ is. It follows that $R(Y,X)X=0$.
\end{proof}

\subsection{Properties of the distance function to a leaf }

Let us fix a point $x \in \M$ and consider the distance function 
\[
r_x(y): =d(y,\mathcal F_{x}), \, y \in \M,
\]
where $\mathcal F_{x}$ is the unique leaf going through $x$. According to \cite[Lemma 4.3]{Hermann} it follows  from the bundle-like property of the foliation that if $y,y' \in  \M$ are in the same leaf then $r_x(y)=r_x(y')$. Therefore, if $r_x$ is differentiable at $y$ then for any vector field $Z \in \Gamma (\V)$, we have $Zr_x (y)=0$.

Since $\M$ is complete the Riemannian exponential map $\exp$ is defined on all of the tangent bundle $T\M$. Its restriction to the normal bundle $T^\perp \mathcal F_{x}$ of $\mathcal F_{x}$ is called the normal exponential map. For $\xi=(y,u) \in T^\perp \mathcal F_{x}$, where $y\in \mathcal F_{x}$ and $u\in \mathcal{H}_y$, the curve 
\[
\gamma(t)=\exp (y, t u ), \, t \ge 0
\]
is therefore a geodesic such that $\gamma(0)=y$ and $\dot\gamma(0)=u$.  Consider then the unique smooth  curve $\sigma$, such that $\sigma(0)=y$, $\dot \sigma (0)=u$ and $\nabla_{\dot \sigma} \dot \sigma=0$. Since $\nabla$ preserves $\mathcal{V}$ and $\dot \sigma (0)_\mathcal{V}=0$, one deduces that $\dot {\sigma}(t) \in \mathcal{H}_{\sigma(t)}$ and thus one also has $\nabla_{\dot\sigma} \dot\sigma+ J_{\dot\sigma} \dot\sigma=0$. By uniqueness of solutions for the geodesic equation,  one must have $\sigma=\gamma$. Therefore, $\dot\gamma_\mathcal{V}=0$ which means that $\gamma$ is horizontal, i.e. everywhere tangent to $\mathcal{H}$, and satisfies moreover
\begin{align}\label{geodesic h}
\nabla_{\dot\gamma} \dot\gamma=0.
\end{align}
In particular, we recovered the fact, true in any bundle-like Riemannian foliation, that if $\gamma$ is a geodesic which is perpendicular to one leaf, then it will be perpendicular to any leaf it intersects. If $|u|=1$, we set
\[
\tau (\xi)=\sup \{ t  \ge 0 : r_x(\gamma(t))=t \} \in [0,\infty].
\]
We always have $\tau (\xi) >0$ and if $\tau (\xi)<+\infty$, then $\gamma(\tau (\xi))$ is called the cut point of $\mathcal F_{x}$ along $\gamma$. The set
\[
\mathbf{Cut}(\mathcal F_{x}):=\left\{ \exp (y,\tau (\xi) u)  : \xi=(y,u) \in T^\perp \mathcal F_{x} , |u|=1, \tau (\xi)<+\infty \right\}
\]
is called the cut locus of $\mathcal F_{x}$. Let us denote 
\[
\mathbf{Seg}(\mathcal F_{x}):=\left\{ (y, tu) \in T^\perp \mathcal F_{x} :    \xi=(y,u) \in T^\perp \mathcal F_{x} , |u|=1, 0\le t <\tau (\xi)  \right\}.
\]
We summarize below some known facts about the regularity of $r_x$, see \cite{MR1941909} and references therein:

\begin{enumerate}
\item In the class of continuous functions bounded from below the distance function $r_x$ is  the unique viscosity solution of the following Hamilton-Jacobi problem:
\begin{align*}
\begin{cases}
| \nabla u| =1 & \text{ in } \M \setminus \mathcal F_{x} \\
u=0 & \text{ on } \mathcal F_{x}.
\end{cases}
\end{align*}
\item The distance function $r_x$ is smooth in $ \M \setminus ( \mathcal F_{x} \cup \mathbf{Cut}(\mathcal F_{x}) )$.
\item The normal exponential map is a smooth diffeomorphism between $\mathbf{Seg}(\mathcal F_{x})$ and $ \M \setminus \mathbf{Cut}(\mathcal F_{x}) $.
\item The set $ \mathbf{Cut}(\mathcal F_{x})$ is closed and has measure zero.
\end{enumerate}

\section{$\mathcal F$-Jacobi fields}

In this section we develop the variational theory of geodesics which are transverse to the leaves. Such theory was initiated in \cite{MR0533065}, see also \cite{tubes} or \cite{Gromoll} and the references therein, but for our purpose we need to work with the connection $\nabla$ and therefore rewrite and advance further the theory using that connection. 

\subsection{$\mathcal F$-Jacobi equation and differential of the normal exponential map}

We start with the following basic definitions.

\begin{definition}
Horizontal curves satisfying \eqref{geodesic h} are called horizontal geodesics. 
\end{definition}

\begin{definition}\label{definit Jacobi}
Let $\gamma:[0,\rho]\to \M$ be a horizontal geodesic and let $V$ be a vector field along $\gamma$. If there exists a one-parameter variation  $\Gamma :[0,\rho] \times (-\varepsilon,\varepsilon) \to \M$ such that:
\begin{enumerate}
\item $\Gamma (t,0)=\gamma(t)$, $\partial_s \Gamma(t,s)_{\mid s=0}=V(t)$,
\item For every $s \in (-\varepsilon,\varepsilon)$, $t \to \Gamma (t,s)$ is a horizontal geodesic,
\end{enumerate}
then $V$ is called a $\mathcal F$-Jacobi field of $\gamma$.
\end{definition}

The next Lemma is the Jacobi equation for $\mathcal{F}$-Jacobi fields.

\begin{lemma}\label{Jacobi equation}
Let $\gamma :[0,\rho] \to \M$ be a horizontal geodesic. A vector field $V$ along $\gamma$ is a $\mathcal F$-Jacobi field if and only if
\begin{align}\label{F Jacobi}
\begin{cases}
\nabla_{\dot\gamma} V_\mathcal{V} +\Tor(V,\dot\gamma)=0\\
\nabla_{\dot\gamma}\nabla_{\dot\gamma} V_\mathcal{H} + R(V_\mathcal{H},\dot\gamma)\dot\gamma = 0,
\end{cases}
\end{align}
where $R$ is the Riemann curvature tensor of the connection $\nabla$ as in \eqref{riem curv}.

\begin{proof}
Let $\Gamma :[0,\rho] \times (-\varepsilon,\varepsilon) \to \M$ be a one-parameter variation associated with $V$. In particular $\Gamma(t,0) = \gamma(t)$, $\Gamma(t,s)$ is a horizontal geodesic for all fixed $s \in (-\varepsilon,\varepsilon)$, and $\partial_s \Gamma(t,0) = V(t)$. Let $T = \Gamma_* \partial_t$ and $S = \Gamma_* \partial_s$. We have $T(t,0) = \dot\gamma(t)$ and $S(t,0) = V(t)$. Using the identity $\nabla_T T = 0$ and the definition of $R$, we deduce
\begin{equation}
0 = \nabla_S \nabla_T T  = R(S,T)T + \nabla_T \nabla_S T = R(S,T)T + \nabla_T (\nabla_T S -\Tor(T,S)).
\end{equation}
Moreover, $T$ is horizontal so that
\[
0=(\nabla_S T)_\mathcal{V}=(\nabla_T S -\Tor(T,S))_\mathcal{V}=\nabla_T S_\mathcal{V} -\Tor (T,S).
\]
We deduce that
\[
R(S,T)T + \nabla_T \nabla_T S_\mathcal{H}=0
\]

Computing the above at $s=0$ and using Lemma \ref{hor-ver} yields the Jacobi equation for $V$. Conversely, if a vector field $V$ along $\gamma$ satisfies \eqref{F Jacobi}, then using standard tools as in \cite[Corollary 1.6.1]{Gromoll},  one can construct a variation of $\gamma$ that satisfies the conditions (i) and (ii) of Definition \ref{definit Jacobi}.
\end{proof}
\end{lemma}

The following corollary is immediate from the Jacobi equation.

\begin{corollary}
Let $\gamma :[0,\rho] \to \M$ be a horizontal geodesic.  The space of $\mathcal{F}$-Jacobi fields along $\gamma$ is $2n+m$-dimensional.
\end{corollary}

Let $x \in \M$. For $\xi=(y,u)$ in the normal bundle $T^\perp \mathcal F_{x}$, the connection $\nabla$ induces a canonical isomorphism  $T_\xi T^\perp \mathcal F_{x} \to \mathcal{V}_y \oplus  \mathcal{H}_y$. For $v \in T_\xi T^\perp \mathcal F_{x}$, we will denote $v_\mathcal{V}+v_\mathcal{H} $ the image of $v$ by this isomorphism.

\begin{lemma}\label{diff expo}
Let $x \in \M$ and $\xi=(y,u) \in T^\perp \mathcal F_{x}$. We have for every $v \in T_\xi T^\perp \mathcal F_{x}$
\[
d \exp_\xi (v)=V(1),
\]
where $V$ is the unique $\mathcal F$-Jacobi field along the horizontal geodesic $\gamma(t)=\exp (y,tu)$ such that $V(0)=v_\mathcal{V}$ and $\nabla_{\dot \gamma }V(0)=v_\mathcal{H}+\Tor (u,v_\mathcal{V})$.
\end{lemma}

\begin{proof}
Let $\zeta: (-\varepsilon,\varepsilon) \to T^\perp \mathcal F_{x}$ be a smooth curve such that $\zeta(0)=\xi$ and  $\dot{\zeta}(0)=v$. We define $\Gamma(t,s)=\exp (t\zeta(s))$, $t \in [0,1]$. $\Gamma$ is then a one-parameter variation of the curve $t\to \Gamma(t,0)=\exp( y,tu)$ and for every $s \in (-\varepsilon,\varepsilon)$, $t \to \Gamma (t,s)$ is a horizontal geodesic. By the definition of the $\mathcal{F}$-Jacobi fields, we have
\[
d \exp_\xi (v)=\partial_s \Gamma(1,s)_{\mid s=0}=V(1)
\]
where $V(t)$ is the $\mathcal{F}$-Jacobi field along $\gamma$ which is given by $V(t):=\partial_s \Gamma(t,s)_{\mid s=0}$. We  first notice that
\[
V(0)=\partial_s \Gamma(0,s)_{\mid s=0}=v_\mathcal{V}.
\]
To compute $\nabla_{\dot \gamma }V(0)$ we denote  $T = \Gamma_* \partial_t$ and $S = \Gamma_* \partial_s$. We have then
\[
\nabla_{\dot \gamma }V(0)=(\nabla_T S)_{\mid s=0, t=0}=(\nabla_S T +\Tor (T,S))_{\mid s=0, t=0}=v_\mathcal{H}+\Tor (u,v_\mathcal{V}).
\]
\end{proof}

\begin{definition}\label{focal}
Let $\xi=(y,u) \in T^\perp \mathcal F_{x}$ and  $\gamma(t)=\exp (y,tu)$. If $s >0$ is such that there exists a non zero $\mathcal{F}$-Jacobi field along $\gamma$ satisfying $V_\mathcal{H}(0)=0$ and $V(s)=0$, then $\gamma(s)$ is called a focal point of $\mathcal{F}_x$ along $\gamma$.
\end{definition}

\begin{lemma}
Let $\gamma :[0,\rho] \to \M$ be a unit-speed horizontal geodesic with $\rho < \tau (\gamma(0),\dot \gamma(0))$.  For every  $X \in T_{\gamma(\rho)} \M $ there exists a  unique  $\mathcal{F}$-Jacobi field $V$ along $\gamma$ such that $V_\mathcal{H}(0) = 0$ and $V(\rho) = X$.
\end{lemma}

\begin{proof}
Since $\rho  < \tau (\gamma(0),\dot \gamma(0))$,  the differential map $d\exp$ is invertible at the point $(\gamma(0),\rho \dot \gamma (0))$. It follows from Lemma \ref{diff expo} that $\gamma(\rho)$ is not a focal point. Let now $\mathfrak{L}$ be the $2n+m$ dimensional space of  $\mathcal{F}$-Jacobi fields along $\gamma$. The map $V \to (V_\mathcal{H}(0),V(\rho))$ is injective since  $\gamma(\rho)$ is not a focal point. It is therefore also an isomorphism.
\end{proof}

\subsection{Hessian of the distance to a leaf}

In the next lemma we compute the Hessian of the distance function to a leaf.

\begin{lemma}\label{Hessian comparison}
Let $\gamma :[0,\rho] \to \M$ be a unit-speed horizontal geodesic with $\rho < \tau (\gamma(0),\dot \gamma(0))$. Let $r(\cdot) = d(\cdot, \mathcal{F}_{\gamma(0)})$ and let $\mathrm{Hess}^{\nabla}$ denote the Hessian with respect to the  connection $\nabla$.
\begin{enumerate}
\item We have at $\gamma (\rho)$, $\mathrm{Hess}^{\nabla}(r)(\gamma'(\rho),\gamma'(\rho))=0$.
\item Let $X \in \mathcal{H}_{\gamma(\rho)} $, with $X \perp \dot\gamma(\rho)$. Then, at $\gamma (\rho)$,
\begin{equation}
\mathrm{Hess}^{\nabla}(r)(X,X) = \int_0^{\rho}(| \nabla_{\dot\gamma} V_\mathcal{H}|^2 - \langle R(V_\mathcal{H},\dot\gamma) \dot\gamma,V_\mathcal{H} \rangle )dt=\left\langle V_\mathcal{H} (\rho),  \nabla_{\dot\gamma} V_\mathcal{H} (\rho) \right\rangle ,
\end{equation}
where, in the integrand, $V$ denotes the   $\mathcal{F}$-Jacobi field along $\gamma$ such that $V_\mathcal{H}(0) = 0$ and $V(\rho) = X$.
\item Let $X \in \mathcal{V}_{\gamma(\rho)} $. Then, we have at $\gamma (\rho)$,  $\mathrm{Hess}^{\nabla}(r)(X,X) = 0.$
\end{enumerate}
\end{lemma}

\begin{proof}

First, the identity $\mathrm{Hess}^{\nabla}(r)(\gamma'(\rho),\gamma'(\rho))=0$ follows from differentiating along $\gamma$ the identity $|\nabla r|^2=1$. Then, let $X \in T_{\gamma(\rho)}\M $, with $X \perp \dot\gamma(\rho)$. Consider a curve $\sigma : (-\varepsilon,\varepsilon) \to \M$ with $\sigma(0) = \gamma(\rho)$, $\dot\sigma(0) = X$ and such that $\nabla_{\dot \sigma} \dot \sigma=0$. For any $s \in (-\varepsilon,\varepsilon)$ let $t\mapsto \Gamma(t,s)$, $t \in [0,\rho]$, be the unique minimizing horizontal geodesic, parametrized with constant speed on $[0,\rho]$, joining $\mathcal{F}_{\gamma(0)}$ with $\sigma(s)$. From \cite[Proposition 4.6, (2)]{MR1941909}, provided that $\varepsilon$ is small enough, this family is well defined since $\rho < \tau (\gamma(0),\dot \gamma(0))$. Let now $S= \Gamma_* \partial_s$ and $T = \Gamma_*\partial_t $. As already noted before, we have $T(t,0) = \dot\gamma(t)$ and $S(t,0) = V(t)$. Let us notice that boundary conditions and Jacobi equation imply that $\langle V(t), \dot\gamma(t)\rangle = 0$ for all $t$ and that we have at $s=0$,  $\langle \nabla_TS,T\rangle =0 $. We also notice that by construction $|T(t,0)|=1$. Then, with  everything  computed at $s=0$, we have
\begin{align*}
\mathrm{Hess}^{\nabla} r (X,X) & = \frac{d^2}{ds^2} \int_0^{\rho} | T| dt = \frac{d}{ds} \int_0^{\rho} \frac{1}{|T|} \langle \nabla_ST,T\rangle dt \\
& = \frac{d}{ds} \int_0^{\rho} \frac{1}{|T|} \langle \nabla_TS +\Tor (S,T) ,T\rangle dt \\
& = \frac{d}{ds} \int_0^{\rho} \frac{1}{|T|} \langle \nabla_TS ,T\rangle dt \\
& = \int_0^{\rho} \left( -\langle \nabla_ S T,T\rangle \langle \nabla_TS,T\rangle + \langle \nabla_S \nabla_T S,T\rangle + \langle \nabla_T S, \nabla_S T\rangle\right) dt \\
& = \int_0^{\rho} \langle \nabla_S \nabla_T S,T\rangle + \langle \nabla_T S, \nabla_T S +\Tor (S,T)\rangle dt\\
& = \int_0^{\rho} \left(\langle \nabla_T S, \nabla_T S +\Tor (S,T)\rangle + \left\langle R(S,T)S,T\right\rangle + \langle \nabla_T \nabla_S S,T\rangle \right) dt\\
& = \langle \nabla_{\dot\sigma}\dot\sigma(0), \dot\gamma(\rho)\rangle +  \int_0^{\rho} \left(\langle \nabla_{\dot\gamma} V, \nabla_{\dot\gamma} V+\Tor (V,\dot\gamma)\rangle - \left\langle R(V,\dot\gamma) \dot\gamma,V\right\rangle \right)dt\\
& = \int_0^{\rho} \left(\langle \nabla_{\dot\gamma} V, \nabla_{\dot\gamma} V+\Tor (V,\dot\gamma)\rangle - \left\langle R(V,\dot\gamma) \dot\gamma,V\right\rangle \right)dt \\
&=\int_0^{\rho}(| \nabla_{\dot\gamma} V_\mathcal{H}|^2 - \langle R(V_\mathcal{H},\dot\gamma) \dot\gamma,V_\mathcal{H} \rangle )dt \\
&=\left\langle V_\mathcal{H} (\rho),  \nabla_{\dot\gamma} V_\mathcal{H} (\rho) \right\rangle,
\end{align*}
where in the last equality we integrated by parts, used the Jacobi equation and the fact that $\nabla$ is a metric connection. The boundary term $\langle \nabla_{\dot\sigma}\dot\sigma(0), \dot\gamma(\rho)\rangle$ comes from
\[
\int_0^{\rho} \langle \nabla_T \nabla_S S,T\rangle dt = \langle \nabla_S S,T\rangle (\rho,0)-\langle \nabla_S S,T\rangle (0,0),
\]
where $\langle \nabla_S S,T\rangle (0,0)=0$ because $s \mapsto \Gamma(0,s)$ is a curve in the leaf $\mathcal{F}_{\gamma(0)}$ and $\nabla$ preserves $\mathcal{V}$.
If $X$ is horizontal we get the stated result and if $X$ is vertical we get zero because in that case $V_\mathcal{H}=0$.

\end{proof}

\section{Cartan-Hadamard type theorem for foliations}

From now on, in addition to Assumption \ref{assump total}, we make the following:

\begin{assumption}\label{negative curv}

\
\begin{enumerate}
\item The leaf space $\M / \mathcal{F}$ is simply connected.
\item  There exists a constant $K \ge 0$ such that for every $x \in \M$ and every $X,Y \in \Ho_x$ that satisfy $X \perp Y$ and $|X|=|Y|=1$, we have
\begin{align}\label{sectional}
\left \langle R(X,Y)Y, X \right\rangle \le -K,
\end{align}
where $R$ is the Riemann curvature tensor of the connection $\nabla$ as in \eqref{riem curv}.
\item The leaves of $\mathcal{F}$ have trivial holonomy.
\end{enumerate}

\end{assumption}

\begin{remark}
The condition \eqref{sectional} can equivalently be rewritten in terms of the sectional curvature of the Levi-Civita connection. Indeed, in the simple situation where the foliation $\mathcal{F}$ comes from a Riemannian submersion $\pi : (\M, g) \to (\mathbb B, h)$ the upper bound  \eqref{sectional} is simply equivalent to the fact that the sectional curvature of $\mathbb B$ be bounded above by $-K$. Now, since any bundle-like Riemannian foliation is locally given by a Riemannian submersion, one can apply the O'Neills formulas for curvature (see \cite[Page 241]{Besse}) to rewrite \eqref{sectional}  using the sectional curvature of the Levi-Civita connection instead. Denoting by $\mathfrak{K}$ the sectional curvature of the Levi-Civita connection on $\M$, the condition \eqref{sectional} is therefore equivalent  to
\[
\mathfrak{K}(X,Y) \le -K -\frac{3}{4} \left|  [X,Y]_\mathcal{V} \right|^2
\]
 at every point $x \in \M$ and for every $X,Y \in \Gamma (\Ho)$ that are orthonormal at $x$.
\end{remark}

We now prove the following theorem.

\begin{theorem}\label{empty cut}
For every $x \in \M$,  $\mathbf{Cut}(\mathcal F_{x})=\emptyset$  and the normal exponential map $T^\perp\mathcal{F}_x \to \M$ is a diffeomorphism.
\end{theorem}

We proceed in several steps and split the proof in different lemmas. Throughout the argument the point $x\in \M$ is fixed.

\begin{lemma}\label{no conjugate}
Let $\gamma:[0,\rho]\to \M$ be a horizontal geodesic started from a point of $\mathcal F_x$. The set of focal points of $\mathcal F_x$ along $\gamma$ is empty.
\end{lemma}

\begin{proof}
Let $V$ be a $\mathcal F$-Jacobi field along $\gamma$ such that $V_\mathcal{H}(0)=0$ and $V(s)=0$. Let $\kappa(t)=| V_\mathcal{H} (t) |^2$. One has $\kappa'(t)=2\left\langle \nabla_{\dot \gamma} V_\mathcal{H}(t), V_\mathcal{H}(t) \right\rangle$ and
\begin{align*}
\kappa''(t)&=2\left\langle \nabla_{\dot \gamma}  \nabla_{\dot \gamma} V_\mathcal{H} (t), V_\mathcal{H} (t) \right\rangle +2 |\nabla_{\dot \gamma} V_\mathcal{H}(t)|^2 \\
 &=-2\left\langle R(V_\mathcal{H}(t),\dot\gamma)\dot\gamma, V_\mathcal{H} (t) \right\rangle +2 |\nabla_{\dot \gamma} V_\mathcal{H}(t)|^2 \ge 0.
\end{align*}
Therefore,  the function $\kappa$ is convex, non-negative and vanishes at 0 and $s$. This implies $\kappa=0$ so that $V_\mathcal{H}=0$. We then deduce from the first equation in  \eqref{F Jacobi} that $V$ solves the linear first order differential equation
\[
\nabla_{\dot\gamma} V_\mathcal{V} +\Tor(V_\mathcal{V} ,\dot\gamma)=0
\]
Since $V_\mathcal{V}$ vanishes at $s$, it must be identically zero along $\gamma$. We conclude that $V$ itself identically vanishes  along $\gamma$. 
\end{proof}

\begin{lemma}\label{local homeo}
The map $\Phi: u \to \mathcal{F}_{\exp (x,u)}$ is a local homeomorphism between the horizontal space  $\mathcal{H}_x$ and the leaf space $\M / \mathcal{F}$.
\end{lemma}

\begin{proof}
By Assumption \ref{negative curv} (iii) and Remark \ref{leaf space manifold}, the leaf space $\M / \mathcal{F}$ is a smooth $n$-dimensional Riemannian manifold and the projection map $\pi : \M \to \M / \mathcal{F}$, $\pi(y)=\mathcal{F}_y$, is a Riemannian submersion such that $\ker d\pi=\mathcal{V}$. It follows from Lemmas \ref{diff expo} and \ref{no conjugate} that the differential map of  $\exp: T^\perp \mathcal{F}_x \to \M $ is of maximal rank and thus invertible at any point $\xi=(y,u) \in T^\perp \mathcal{F}_x$. Moreover, for every $v \in T_\xi T^\perp \mathcal F_{x}$ such that $v_\mathcal{H}=0$ one has $d \exp_\xi (v) \in \mathcal{V}_{\exp(y,u)}$. Indeed, from the Jacobi equation, the unique $\mathcal F$-Jacobi field along the horizontal geodesic $\gamma(t)=\exp (y,tu)$ such that $V(0)=v_\mathcal{V}$ and $\nabla_{\dot \gamma }V(0)=\Tor( u,v_\V)$ has to be identically vertical along $\gamma$. Counting dimensions, we deduce that
\[
d \exp_\xi \left( \left\{ v \in T_\xi T^\perp \mathcal F_{x} : v_\mathcal{H}=0 \right\} \right)=\mathcal{V}_{\exp(y,u)}=\ker d\pi_{\exp(y,u)},
\]
Since $d \exp_\xi$ is injective, it follows that for every $v \in T_\xi T^\perp \mathcal F_{x}$,
\[
v_\mathcal{V}=0 \text{ and } d\pi_{\exp(y,u)} \left( d \exp_\xi (v) \right)=0 \implies v=0.
\]

Let now $\phi: \mathcal{H}_x \to \M$, $\phi(u)= \exp (x,u)$, so that $\Phi=\pi \circ \phi$. For $u,w \in \mathcal{H}_x$ one has $d\phi_u(w)=d \exp_{(x,u)} (v)$, where $v \in T_{(x,u)} T^\perp \mathcal F_{x}$ is such that $v_\mathcal{V}=0$ and $v_\mathcal{H}=w$. From the previous discussion, the map
\[
d\Phi_u=d\pi_{\phi(u)} \circ d\phi_u : \mathcal{H}_x \to T_{\Phi(u)} (\M / \mathcal{F})
\]
is therefore injective, and thus bijective because both spaces have dimension $n$. One concludes from the inverse function theorem that $\Phi$ is a local diffeomorphism, and in particular a local homeomorphism, between $\mathcal{H}_x$ and $\M / \mathcal{F}$.
\end{proof}

We recall that $(\M / \mathcal{F}, d_{\M / \mathcal{F}})$ is a geodesic metric space, see Lemma \ref{geodesic space}.  The local homeomorphism  $\Phi$ of Lemma \ref{local homeo} allows us to pull-back to $\mathcal{H}_x$ the distance $d_{\M / \mathcal{F}}$. More precisely, for $ u,v \in \mathcal{H}_x$ we define the following distance:
\[
d_{\mathcal{H}_x}(u,v)=\inf \left\{ L( \Phi  \circ \sigma): \sigma:[0,1] \to \mathcal{H}_x, \, \sigma \text{ continuous} ,\sigma(0)=u,\sigma(1)=v   \right\}.
\]

By Assumption \ref{negative curv} (iii) and Remark \ref{leaf space manifold}, $d_{\M / \mathcal{F}}$ is the Riemannian distance of a Riemannian metric $h$ on $\M / \mathcal{F}$ and, by the proof of Lemma \ref{local homeo}, $\Phi$ is a local diffeomorphism. Therefore, $\Phi$ is a local isometry from $(\mathcal{H}_x, \Phi^* h)$ to $(\M / \mathcal{F},h)$, so that for every continuous curve $\sigma$ in $\mathcal{H}_x$, $L( \Phi \circ \sigma)$ is the length of $\sigma$ for the Riemannian distance of $\Phi^* h$. It follows that $d_{\mathcal{H}_x}$ is the Riemannian distance of the Riemannian metric $\Phi^* h$ on $\mathcal{H}_x$. In particular, $d_{\mathcal{H}_x}$ is a distance which induces the usual topology of $\mathcal{H}_x$.

\begin{lemma}\label{complete proof}

The following hold:

\begin{enumerate}
\item $\Phi: (\mathcal{H}_x, d_{\mathcal{H}_x}) \to (\M / \mathcal{F}, d_{\M / \mathcal{F}}) $ is a local isometry;
\item $(\mathcal{H}_x, d_{\mathcal{H}_x})$ is a length space, i.e. for every $ u,v \in \mathcal{H}_x$ 
\[
d_{\mathcal{H}_x}(u,v)=\inf \left\{ L(  \sigma): \sigma:[0,1] \to \mathcal{H}_x, \, \sigma \text{ continuous} ,\sigma(0)=u,\sigma(1)=v   \right\};
\]

\item For every $u \in \mathcal{H}_x$, $d_{\mathcal{H}_x}(0,u)= |u |$;
\item $(\mathcal{H}_x, d_{\mathcal{H}_x})$ is a complete metric space.
\end{enumerate}

\end{lemma}

\begin{proof}

(i). Let $u \in \mathcal{H}_x$. Let  $V$ be a neighborhood of $u$ in $\mathcal{H}_x$ and $W$ be a neighborhood of $\Phi (u)$ in $\M / \mathcal{F}$ such that $\Phi$ is an homeomorphism from $V$ onto $W$. We denote by $\Phi_V^{-1} : W \to V$ the inverse of the restriction of $\Phi$ to $V$. Let $\varepsilon>0$ such that $B_{\M/\mathcal{F}} (\Phi(u), 3\varepsilon) \subset W$ where $B_{\M/\mathcal{F}}$ denotes the open metric ball in $\M/\mathcal{F}$. Let now $p_1,p_2 \in B_{\M/\mathcal{F}} (\Phi(u), \varepsilon)$, $p_1\neq p_2$. Since $(\M / \mathcal{F}, d_{\M / \mathcal{F}})$ is a geodesic metric space, there exists a unit-speed minimizing geodesic $\gamma:[0,\rho] \to \M / \mathcal{F}$ connecting $p_1$ to $p_2$ where $\rho=d_{\M / \mathcal{F}}(p_1,p_2)$. This curve $\gamma$ is included in $W$ because
\[
d_{\M / \mathcal{F}}(\Phi(u), \gamma(t)) \le d_{\M / \mathcal{F}}(\Phi(u), p_1)+d_{\M / \mathcal{F}}(p_1, \gamma(t))< \varepsilon +t \le 3\varepsilon.
\]
We consider then the curve $\sigma(t)=\Phi_V^{-1}(\gamma(\rho t)) \in V$. It connects $\Phi_V^{-1}(p_1)$ to $\Phi_V^{-1}(p_2)$, thus
\[
d_{\mathcal{H}_x} (\Phi_V^{-1}(p_1) , \Phi_V^{-1}(p_2) ) \le L (\Phi \circ \sigma)=\rho=d_{\M / \mathcal{F}}(p_1,p_2).
\]
On the other hand, for any  continuous curve  $\sigma:[0,1] \to \mathcal{H}_x$ such that  $\sigma(0)=\Phi_V^{-1}(p_1),\sigma(1)=\Phi_V^{-1}(p_2)$, one has $ L (\Phi \circ \sigma) \ge d_{\M / \mathcal{F}}(p_1,p_2)$. Therefore $d_{\mathcal{H}_x} (\Phi_V^{-1}(p_1) , \Phi_V^{-1}(p_2) )\ge d_{\M / \mathcal{F}}(p_1,p_2)$. One concludes that $\Phi$ is an isometry between the sets $\Phi_V^{-1}(B_{\M/\mathcal{F}} (\Phi(u), \varepsilon))$ and $B_{\M/\mathcal{F}} (\Phi(u), \varepsilon)$.

\

 \noindent (ii). This follows easily from (i), since $\Phi$ being a local isometry implies $L( \Phi  \circ \sigma)=L(\sigma)$.

\

\noindent (iii).  From the proof of Lemma \ref{local homeo} and (i) it follows that for $\eta >0$ small enough,  $u\to\exp (x,u)$ induces a diffeomorphism from $B_{\mathcal{H}_x}(0,2\eta)$ onto its image and that for $u \in B_{\mathcal{H}_x}(0,2\eta)$, $d_{\mathcal{H}_x}(0,u)= |u |$. Let now $u \in \mathcal{H}_x$ with $d_{\mathcal{H}_x}(0,u) \ge 2\eta$. Let us denote $S_{\mathcal{H}_x}(0,\eta)=\left\{ v \in \mathcal{H}_x , |v|=\eta \right\} $ and consider $w \in S_{\mathcal{H}_x}(0,\eta)$ such that
\[
d_{\mathcal{H}_x}(w,u)=\min_{v \in S_{\mathcal{H}_x}(0,\eta)} d_{\mathcal{H}_x}(v,u).
\]
We denote $\gamma(t)=\frac{t}{\eta} w$, $t \ge 0$ and consider the set 
\[
S=\left\{  s \in [\eta, d_{\mathcal{H}_x}(0,u) ]: d_{\mathcal{H}_x}(\gamma(s),u)= d_{\mathcal{H}_x}(0,u)-s  \right\}. 
\]
Proving that $d_{\mathcal{H}_x}(0,u) \in S$ will conclude the argument because $\gamma(d_{\mathcal{H}_x}(0,u))=u$ implies $\frac{d_{\mathcal{H}_x}(0,u)}{\eta} w=u$ and therefore $d_{\mathcal{H}_x}(0,u)=|u|$. First, we note that $\eta \in S$. Indeed, since $(\mathcal{H}_x, d_{\mathcal{H}_x})$ is a length space one has
\begin{align*}
d_{\mathcal{H}_x}(0,u)=\inf_{v \in S_{\mathcal{H}_x}(0,\eta)} ( d_{\mathcal{H}_x}(0,v)+d_{\mathcal{H}_x}(v,u)) =\eta+d_{\mathcal{H}_x}(w,u).
\end{align*}

Then, we remark that $S$ is a closed set, so that $r:=\sup S \in S$. Let us assume, for the sake of contradiction, that $r<d_{\mathcal{H}_x}(0,u)$. Let then $\varepsilon >0$ such that $r+\varepsilon < d_{\mathcal{H}_x}(0,u)$ and $\hat{w} \in \mathcal{H}_x$ such that
\[
d_{\mathcal{H}_x}(\hat{w},u)=\min \left\{  d_{\mathcal{H}_x}(v,u) : d_{\mathcal{H}_x}(v,\gamma(r))=\varepsilon \right\}. 
\]
As before, one has $d_{\mathcal{H}_x}(\gamma(r),u)=\varepsilon+d_{\mathcal{H}_x}(\hat w,u)$. Therefore, since $r \in S$,  one has $\varepsilon+d_{\mathcal{H}_x}(\hat w,u)=d_{\mathcal{H}_x}(0,u)-r$. By the  triangle inequality we  have
\begin{align}\label{triangle proof}
d_{\mathcal{H}_x}(0,\hat w)\ge d_{\mathcal{H}_x}(0,u)-d_{\mathcal{H}_x}(u,\hat w)=r +\varepsilon.
\end{align}
Now, from (i), Assumption \ref{negative curv} (iii) and Remark \ref{leaf space manifold}, $(\mathcal{H}_x, d_{\mathcal{H}_x})$ is locally isometric to a Riemannian manifold. Therefore, provided that $\varepsilon $ is small enough, there exists a unique length minimizing unit-speed geodesic $\sigma$ connecting $\gamma(r)$ to $\hat{w}$. We denote by $\hat{\gamma}$ the concatenation of $\gamma_{\big | [0,r]}$ and $\sigma$. The curve $\hat{\gamma}$ has length $r+\varepsilon$ and connects $0$ and $\hat w$, thus from (ii), $d_{\mathcal{H}_x}(0,\hat w)\le r +\varepsilon$. Combining this with \eqref{triangle proof} gives that $d_{\mathcal{H}_x}(0,\hat w)=r+\varepsilon$ so that $\hat{\gamma}$ is a unit-speed minimizing geodesic.  Since $(\mathcal{H}_x, d_{\mathcal{H}_x})$ is locally isometric to a Riemannian manifold, unit-speed minimizing geodesics in $\mathcal{H}_x$ can not be branching, therefore $\hat{\gamma} =\gamma_{\big | [0,r+\varepsilon]}$ and $r+\varepsilon \in S$, which is a contradiction.

\

\noindent (iv). It is enough to prove that bounded and closed sets are compact. Let $K \subset \mathcal{H}_x$ be bounded and closed set. According to (iii), for $C>0$ large enough we have $K \subset \left\{ u \in \mathcal{H}_x, |u| \le C \right\}$ which is a compact subset of $\mathcal{H}_x$. It follows that $K$ is a closed subset of a compact set and as such, is compact.

\end{proof}

\begin{lemma}\label{global iso}
The map $\Phi: (\mathcal{H}_x, d_{\mathcal{H}_x}) \to (\M / \mathcal{F}, d_{\M / \mathcal{F}}) $ is an isometry.
\end{lemma}

\begin{proof}
From Lemmas \ref{local homeo} and \ref{complete proof}, and since by Assumption \ref{negative curv} (iii) and Remark \ref{leaf space manifold} the metric space $(\M / \mathcal{F}, d_{\M / \mathcal{F}})$ is a Riemannian manifold equipped with its Riemannian distance, the following properties are satisfied:
\begin{enumerate}
\item $(\mathcal{H}_x, d_{\mathcal{H}_x})$ is complete;
\item $\Phi: (\mathcal{H}_x, d_{\mathcal{H}_x}) \to (\M / \mathcal{F}, d_{\M / \mathcal{F}}) $ is a local isometry;
\item For every $p \in \M / \mathcal{F}$, there exists $\varepsilon >0$ such that for every $q \in B_{\M / \mathcal{F}} (p,\varepsilon)$ there exists a unique unit-speed minimizing geodesic $\sigma_q:[0,\rho] \to B_{\M / \mathcal{F}} (p,\varepsilon)$ joining $p$ to $q$, and $\sigma_q$ varies continuously with $q$.
\end{enumerate}
It then  follows from classical arguments in the theory of covering spaces and originating from W. Ambrose \cite[Theorem A]{Ambrose} that $\Phi$ is a covering map, see also \cite{Jaramillo} for further details. Since $\M / \mathcal{F}$ is simply connected, it follows that $\Phi$ is an isometry.
\end{proof}

We are now ready for the proof of Theorem \ref{empty cut}

\

\noindent
\textit{Proof of Theorem \ref{empty cut}.} It follows from Lemmas \ref{global iso} and \ref{complete proof} that for every $u \in \mathcal{H}_x$, $d_{\mathcal{H}_x}(0,u)=d_{\M / \mathcal{F}}( \mathcal{F}_x, \mathcal{F}_{\exp(x,u) } )=|u|$. From \cite[Lemma 4.3]{Hermann}, this is equivalent to the fact that $d( \mathcal{F}_x, \exp(x,u)  )=|u|$. Therefore, one has $\mathbf{Cut}(\mathcal F_{x})=\emptyset$. It remains to prove that $\M$ is diffeomorphic to $T^\perp\mathcal{F}_x$. Since $\mathbf{Cut}(\mathcal F_{x})=\emptyset$, according to Proposition 4.5  and Remark 3.6 in \cite{MR1941909}, for every $z$, there exists a unique horizontal, unit-speed, and distance minimizing geodesic connecting $z$ to $\mathcal{F}_x$. Thus, the normal exponential map $T^\perp\mathcal{F}_x \to \M$ is a diffeomorphism. \qed

\section{Laplacian comparison theorems and bottom of the spectrum}

From now on, in addition to Assumptions \ref{assump total} and \ref{negative curv} we make the following:

\begin{assumption}
All the  leaves are minimal submanifolds of $\M$.
\end{assumption}

\subsection{Comparison results}

We  denote by $\Delta$ the Laplace-Beltrami operator of $\M$ and by $\Delta_\mathcal{H}$  the horizontal Laplacian of the foliation. We refer to \cite{BaudoinEMS2014} for the definition of $\Delta_\mathcal{H}$.

Let $X_1,\cdots,X_n$ be an horizontal orthonormal frame and  $Z_1,\cdots,Z_m$ be a vertical orthonormal frame around a point $x \in \M$. We note that, by definition, the mean curvature vector of a leaf  is the horizontal part of $\sum_{j=1}^m D_{Z_j}Z_j$  where, as before, $D$ denotes the  Levi-Civita connection.  The assumption that all the leaves are minimal submanifolds of $\M$ is therefore equivalent to the fact that the vector $\sum_{j=1}^m D_{Z_j}Z_j$ is always vertical. From the definition of $\nabla$, this is also equivalent to
\[
\sum_{j=1}^m D_{Z_j}Z_j=\sum_{j=1}^m \nabla_{Z_j}Z_j.
\]
Also note that from the bundle-like condition \eqref{bundle-like}, $D_{X_i}X_i$ is horizontal so that $D_{X_i}X_i=\nabla_{X_i}X_i $. 
Denoting now by $\mathrm{Hess}^{D}$ the Hessian for the Levi-Civita connection $D$, one has  for any smooth function $f:\M \to \mathbb R$
\begin{align*}
\Delta f(x) &=\sum_{i=1}^n \mathrm{Hess}^{D} f (X_i,X_i) +\sum_{j=1}^m \mathrm{Hess}^{D} f (Z_j,Z_j) \\
 &=\sum_{i=1}^n (X_i^2 f -D_{X_i}X_i f) +\sum_{j=1}^m (Z_j^2 f -D_{Z_j}Z_j f)\\
 &=\sum_{i=1}^n (X_i^2 f -\nabla_{X_i}X_i f) +\sum_{j=1}^m (Z_j^2 f -\nabla_{Z_j}Z_j f) \\
 &=\sum_{i=1}^n \mathrm{Hess}^{\nabla} f (X_i,X_i) +\sum_{j=1}^m \mathrm{Hess}^{\nabla} f (Z_j,Z_j)
\end{align*}
and
\[
\Delta_\mathcal{H} f(x) =\sum_{i=1}^n \mathrm{Hess}^{\nabla} f (X_i,X_i).
\]

We can now state the Laplacian comparison theorem.

\begin{lemma}\label{comparison laplace}
Let $x \in \M$ and consider the distance function 
\[
r_x(y): =d(y,\mathcal F_{x}), \, y \in \M.
\]
Then $r_x$ is smooth on $\M \setminus \mathcal F_{x}$ and for every $y \in \M \setminus \mathcal F_{x}$,
\begin{align*}
\Delta_{\mathcal{H}} r_x(y)=\Delta r_x(y)
 \ge   \begin{cases} 
\frac{n-1}{r_x(y)} & \text{if $K = 0$,}\\ (n-1)\sqrt{K} \coth (\sqrt{K} r_x(y) )& \text{if $K > 0$.} \end{cases}
\end{align*}
\end{lemma}

\begin{proof}
From Theorem \ref{empty cut}, the cut-locus of $\mathcal F_{x}$ is empty which implies that for every $y \in  \M \setminus \mathcal F_{x}$ there is a unique unit-speed horizontal, and minimizing geodesic $\gamma:[0,\rho] \to \M$ connecting $\mathcal F_{x}$ to $y$. It also implies that  $r_x$ is smooth at $y$. Consider an orthonormal frame of $\mathcal{H}_y$  given by $\left\{  \gamma'(\rho), X_1,\cdots,X_{n-1} \right\}$ and an orthonormal frame of $\mathcal{V}_y$ given by $\left\{  Y_1,\cdots,Y_{m} \right\}$. Applying Lemma \ref{Hessian comparison} and the computations just before Lemma \ref{comparison laplace} yields
\begin{align*}
\Delta_{\mathcal{H}} r_x(y)=\Delta r_x(y)&=\mathrm{Hess}^{\nabla}(r)(\gamma'(\rho),\gamma'(\rho))+\sum_{i=1}^{n-1} \mathrm{Hess}^{\nabla}(r)(X_i,X_i) \\
 &=\sum_{i=1}^{n-1} \left\langle V_{i,\mathcal{H}} (\rho),  \nabla_{\dot\gamma} V_{i,\mathcal{H}} (\rho) \right\rangle,
\end{align*}
where $V_i$ is the   $\mathcal{F}$-Jacobi field along $\gamma$ such that $V_{i,\mathcal{H}}(0) = 0$ and $V_i(\rho) = X_i$. We note that, from \eqref{F Jacobi}, the second derivative of $\left\langle V_{i,\mathcal{H}} ,\dot\gamma \right\rangle$ is $-\left\langle R(V_{i,\mathcal{H}},\dot\gamma)\dot\gamma ,\dot\gamma \right\rangle=0$. Since $\left\langle V_{i,\mathcal{H}} ,\dot\gamma \right\rangle$ vanishes at $t=0$ and $t=\rho$, we deduce that $V_{i,\mathcal{H}}(t) \perp \dot\gamma(t)$ for every $t \in [0,\rho]$. Moreover, $V_{i,\mathcal{H}}(t) \neq 0$ for every $t \in (0,\rho]$: otherwise, the argument in the proof of Lemma \ref{no conjugate} would give $V_{i,\mathcal{H}}=0$ identically, which contradicts $V_{i,\mathcal{H}}(\rho)=X_i$. The next  step in the proof is then a variation of  classical techniques  (see for instance the proof of the Rauch comparison theorem \cite[Theorem 6.4.3]{Petersen}). Indeed, consider the ratio
\[
\kappa_i(t)=\frac{\left\langle V_{i,\mathcal{H}} (t),  \nabla_{\dot\gamma} V_{i,\mathcal{H}} (t) \right\rangle}{| V_{i,\mathcal{H}} (t)|^2}.
\]
We have  
\begin{align*}
\dot\kappa_i(t)&=\frac{ | \nabla_{\dot\gamma} V_{i,\mathcal{H}} (t) |^2| V_{i,\mathcal{H}} (t)|^2 + \left\langle V_{i,\mathcal{H}} (t), \nabla_{\dot\gamma}  \nabla_{\dot\gamma} V_{i,\mathcal{H}} (t) \right\rangle | V_{i,\mathcal{H}} (t)|^2-2\left\langle V_{i,\mathcal{H}} (t),  \nabla_{\dot\gamma} V_{i,\mathcal{H}} (t) \right\rangle^2}{| V_{i,\mathcal{H}} (t)|^4} \\
 &=\frac{ | \nabla_{\dot\gamma} V_{i,\mathcal{H}} (t) |^2| V_{i,\mathcal{H}} (t)|^2 - \left\langle V_{i,\mathcal{H}} (t), R(V_{i,\mathcal{H}}(t),\dot\gamma)\dot\gamma \right\rangle | V_{i,\mathcal{H}} (t)|^2-2\left\langle V_{i,\mathcal{H}} (t),  \nabla_{\dot\gamma} V_{i,\mathcal{H}} (t) \right\rangle^2}{| V_{i,\mathcal{H}} (t)|^4} \\
 & \ge \frac{ | \nabla_{\dot\gamma} V_{i,\mathcal{H}} (t) |^2| V_{i,\mathcal{H}} (t)|^2 +K  | V_{i,\mathcal{H}} (t)|^4-2\left\langle V_{i,\mathcal{H}} (t),  \nabla_{\dot\gamma} V_{i,\mathcal{H}} (t) \right\rangle^2}{| V_{i,\mathcal{H}} (t)|^4} \\
 &\ge K-\kappa_i(t)^2.
\end{align*}
Since $V_{i,\mathcal{H}}(0)=0$ and $V_{i,\mathcal{H}} \neq 0$, one has $\nabla_{\dot\gamma} V_{i,\mathcal{H}}(0) \neq 0$ and therefore $\lim_{t \to 0^+} t\kappa_i(t)=1$. We then conclude by standard comparison theory of  Riccati type differential equations that
\[
\kappa_i(t)\ge   \begin{cases} 
\frac{1}{t} & \text{if $K = 0$,}\\ \sqrt{K} \coth (\sqrt{K} t) & \text{if $K > 0$.} \end{cases}
\]
This yields:
\[
\left\langle V_{i,\mathcal{H}} (\rho),  \nabla_{\dot\gamma} V_{i,\mathcal{H}} (\rho) \right\rangle=\kappa_i(\rho)\ge   \begin{cases} 
 \frac{1}{\rho} & \text{if $K = 0$,}\\ \sqrt{K} \coth ( \sqrt{K} \rho ) & \text{if $K > 0$,} \end{cases}
\]
which concludes the proof.
\end{proof}

\subsection{Bottom of the spectrum}

We now conclude the paper with the proof of the estimate in Theorem \ref{main theorem} for the bottom of spectrum for both the Laplacian and horizontal Laplacian of the foliation. This follows from the following inequality where we denote by $C^\infty_c(\M)$ the space of smooth and compactly supported functions on $\M$ and by $m$ the Riemannian volume measure.

Here, the bottom of the spectrum of $-\Delta_\mathcal{H}$ is understood as the bottom of the spectrum of the Friedrichs extension of the non-negative symmetric operator $-\Delta_\mathcal{H}$ defined on $C^\infty_c(\M)$, that is
\[
\inf \left\{ \frac{\int_\M | \nabla_\Ho f | ^2 dm}{\int_\M f^2 dm} : f \in C^\infty_c(\M), f \neq 0 \right\},
\]
and we do not assume that $\mathcal{H}$ is bracket generating. We also recall that $n \ge 2$ denotes the codimension of $\mathcal{F}$, which is denoted by $d$ in Theorem \ref{main theorem}.

\begin{theorem}\label{poincare inequality}
For every $f \in C^\infty_c(\M)$,
\[
\frac{(n-1)^2 K}{4} \int_\M f^2 dm \le  \int_\M | \nabla_\Ho f | ^2 dm \le  \int_\M | \nabla f | ^2 dm.
\]
\end{theorem}

\begin{proof}
The inequality $ \int_\M | \nabla_\Ho f | ^2 dm \le  \int_\M | \nabla f | ^2 dm$ is straightforward since $\nabla_\Ho f $ is the orthogonal projection of $\nabla f $ onto $\Ho$. Then, we can assume $K > 0$ otherwise there is nothing to prove. Let $f \in C^\infty_c(\M)$ which is not zero. From Lemma \ref{complete proof} (iii) and Lemma \ref{global iso},  the leaf space $\M / \mathcal{F}$  is not compact and therefore we can find $x \in \M$ such that the leaf $\mathcal F_x$ does not intersect the support of $f$. As before, we denote by $r_x$ the distance to $\mathcal F_x$. From Lemma \ref{comparison laplace} and an integration by parts formula justified by the fact that  $r_x$ is smooth on the support of $f$ we have:

\begin{align*}
(n-1) \sqrt{K} \int_\M f^2 dm \le \int_\M f^2 (\Delta_\Ho r_x )dm=-\int_\M \langle  \nabla_\Ho f^2, \nabla_\Ho r_x \rangle dm.
\end{align*}

Now, from Cauchy-Schwarz inequality and the fact that $| \nabla_\Ho r_x|=1 $, we have
\[
-\int_\M \langle  \nabla_\Ho f^2, \nabla_\Ho r_x \rangle dm \le \int_\M |  \nabla_\Ho f^2 | dm \le  2 \left( \int_\M f^2 dm \right)^{1/2} \left( \int_\M |  \nabla_\Ho f |^2 dm \right)^{1/2}.
\]
We conclude 
\begin{align*}
(n-1) \sqrt{K} \int_\M f^2 dm \le  2 \left( \int_\M f^2 dm \right)^{1/2} \left( \int_\M |  \nabla_\Ho f |^2 dm \right)^{1/2},
\end{align*}
which yields the expected result.
\end{proof}


\begin{thebibliography}{10}

\bibitem{Ambrose}
W.~Ambrose.
\newblock Parallel translation of {R}iemannian curvature.
\newblock {\em Ann. of Math. (2)}, 64:337--363, 1956.

\bibitem{MR4617797}
Werner Ballmann, Mayukh Mukherjee, and Panagiotis Polymerakis.
\newblock On the spectrum of certain {H}adamard manifolds.
\newblock {\em SIGMA Symmetry Integrability Geom. Methods Appl.}, 19:Paper No.
  050, 19, 2023.

\bibitem{BaudoinEMS2014}
Fabrice Baudoin.
\newblock Sub-{L}aplacians and hypoelliptic operators on totally geodesic
  {R}iemannian foliations.
\newblock In {\em Geometry, analysis and dynamics on sub-{R}iemannian
  manifolds. {V}ol. 1}, EMS Ser. Lect. Math., pages 259--321. Eur. Math. Soc.,
  Z\"urich, 2016.

\bibitem{baudoin2019comparison}
Fabrice Baudoin, Erlend Grong, Luca Rizzi, and Sylvie Vega-Molino.
\newblock Comparison theorems on {H}-type sub-{R}iemannian manifolds.
\newblock {\em Calc. Var.}, 64, 2025.

\bibitem{Besse}
Arthur~L. Besse.
\newblock {\em Einstein manifolds}, volume~10 of {\em Ergebnisse der Mathematik
  und ihrer Grenzgebiete (3)}.
\newblock Springer-Verlag, Berlin, 1987.

\bibitem{BridsonHaefliger1999}
Martin~R. Bridson and Andr\'{e} Haefliger.
\newblock {\em Metric Spaces of Non-Positive Curvature}, volume 319 of {\em
  Grundlehren der mathematischen Wissenschaften}.
\newblock Springer, 1999.

\bibitem{CavalcanteManfio2018}
Marcos~P. Cavalcante and Fernando Manfio.
\newblock On the fundamental tone of immersions and submersions.
\newblock {\em Proceedings of the American Mathematical Society},
  146(7):2963--2971, 2018.

\bibitem{Escobales}
Richard~H. Escobales, Jr.
\newblock Sufficient conditions for a bundle-like foliation to admit a
  {R}iemannian submersion onto its leaf space.
\newblock {\em Proc. Amer. Math. Soc.}, 84(2):280--284, 1982.

\bibitem{GallotHulinLafontaine2004}
Sylvestre Gallot, Dominique Hulin, and Jacques Lafontaine.
\newblock {\em Riemannian Geometry}.
\newblock Universitext. Springer, 3rd edition, 2004.

\bibitem{tubes}
Alfred Gray.
\newblock {\em Tubes}, volume 221 of {\em Progress in Mathematics}.
\newblock Birkh\"{a}user Verlag, Basel, second edition, 2004.
\newblock With a preface by Vicente Miquel.

\bibitem{Gromoll}
Detlef Gromoll and Gerard Walschap.
\newblock {\em Metric foliations and curvature}, volume 268 of {\em Progress in
  Mathematics}.
\newblock Birkh\"{a}user Verlag, Basel, 2009.

\bibitem{Gromov1987}
Mikhail Gromov.
\newblock Hyperbolic groups.
\newblock In S.~M. Gersten, editor, {\em Essays in Group Theory}, volume~8 of
  {\em Mathematical Sciences Research Institute Publications}, pages 75--263.
  Springer, 1987.

\bibitem{Jaramillo}
Olivia Gut\'{u} and Jes\'{u}s~A. Jaramillo.
\newblock Global homeomorphisms and covering projections on metric spaces.
\newblock {\em Math. Ann.}, 338(1):75--95, 2007.

\bibitem{Hebda}
James~J. Hebda.
\newblock Curvature and focal points in {R}iemannian foliations.
\newblock {\em Indiana Univ. Math. J.}, 35:321--331, 1986.

\bibitem{MR0533065}
Ernst Heintze and Hermann Karcher.
\newblock A general comparison theorem with applications to volume estimates
  for submanifolds.
\newblock {\em Ann. Sci. \'{E}cole Norm. Sup. (4)}, 11(4):451--470, 1978.

\bibitem{Hermann}
Robert Hermann.
\newblock On the differential geometry of foliations.
\newblock {\em Ann. of Math. (2)}, 72:445--457, 1960.

\bibitem{Hladky}
Robert~K. Hladky.
\newblock Connections and curvature in sub-{R}iemannian geometry.
\newblock {\em Houston J. Math.}, 38(4):1107--1134, 2012.

\bibitem{MR1941909}
Carlo Mantegazza and Andrea~Carlo Mennucci.
\newblock Hamilton-{J}acobi equations and distance functions on {R}iemannian
  manifolds.
\newblock {\em Appl. Math. Optim.}, 47(1):1--25, 2003.

\bibitem{McKean}
H.~P. McKean.
\newblock An upper bound to the spectrum of {$\Delta $} on a manifold of
  negative curvature.
\newblock {\em J. Differential Geometry}, 4:359--366, 1970.

\bibitem{Petersen}
Peter Petersen.
\newblock {\em Riemannian geometry}, volume 171 of {\em Graduate Texts in
  Mathematics}.
\newblock Springer, New York, second edition, 2006.

\bibitem{Polymerakis}
Panagiotis Polymerakis.
\newblock Spectral estimates and discreteness of spectra under {R}iemannian
  submersions.
\newblock {\em Ann. Global Anal. Geom.}, 57(2):349--363, 2020.

\bibitem{MR0107279}
Bruce~L. Reinhart.
\newblock Foliated manifolds with bundle-like metrics.
\newblock {\em Ann. of Math. (2)}, 69:119--132, 1959.

\bibitem{Reinhart2}
Bruce~L. Reinhart.
\newblock Closed metric foliations.
\newblock {\em Michigan Math. J.}, 8:7--9, 1961.

\bibitem{Tondeur}
Philippe Tondeur.
\newblock {\em Foliations on {R}iemannian manifolds}.
\newblock Universitext. Springer-Verlag, New York, 1988.

\end{thebibliography}
\end{document}